\documentclass[12pt,letterpaper]{amsart}
\usepackage[total={6.9in, 9.1in},centering]{geometry}
\usepackage[colorlinks,allcolors=blue]{hyperref} 
\usepackage{xpatch} 
\xpatchcmd{\proof} 
{\itshape} 
{\bfseries} 
{}
{}
\usepackage{mathpazo} 
\usepackage{mathrsfs,latexsym,url,setspace,xcolor}
\newtheorem{theorem}{Theorem}
\newtheorem{proposition}{Proposition}
\newtheorem{lemma}{Lemma}
\newtheorem{corollary}{Corollary}
\theoremstyle{remark}

\newtheorem{definition}{Definition}

\newcommand{\C}{\mathbb{C}}
\newcommand{\D}{\Omega}

\newcommand{\re}{\text{Re}}

\onehalfspace

\title{Zero products of Toeplitz operators on Reinhardt domains}
\author{\v{Z}eljko \v{C}u\v{c}kovi\'c}
\author{Zhenghui Huo}
\author{S\"{o}nmez \c{S}ahuto\u{g}lu}
\email{Zeljko.Cuckovic@utoledo.edu, Zhenghui.Huo@utoledo.edu, 
Sonmez.Sahutoglu@utoledo.edu}
\address{University of Toledo, Department of Mathematics \& Statistics, 
Toledo, OH 43606, USA}
\date{\today}

\begin{document}
\begin{abstract}
Let $\Omega$ be a bounded Reinhardt domain in $\mathbb{C}^n$ and  
$\phi_1,\ldots,\phi_m$ be finite sums of bounded quasi-homogeneous 
functions. We show that if the product of Toeplitz operators 
$T_{\phi_m}\cdots T_{\phi_1}=0$ on the Bergman space on $\Omega$, 
then $\phi_j=0$  for some $j$.
\end{abstract}
\maketitle


\section{Introduction}
Algebraic properties of Toeplitz operators acting on Bergman spaces have 
been intensively studied for the past 30 years.  These problems attracted the 
interests of many operator theorists, partly because they are easy to state but 
quite difficult to solve.  The same problems on the classical Hardy space of the 
unit disk $\mathbb{D}$ were solved by Brown and Halmos in their famous 
paper \cite{BrownHalmos63/64}. A problem that is particularly appealing is 
the zero product problem. Brown and Halmos showed that if the product 
of two Toeplitz operators $T_fT_g = 0$, then either $f=0$ or $g=0$.  In that case, 
we say that the zero product problem has a trivial solution.  The 
question about the zero product of finitely many Toeplitz operators acting 
on the Hardy space was solved much later by Aleman and Vukoti\'c 
\cite{AlemanVukotic09} where they showed the problem also has 
only a trivial solution.

For the Bergman space on $\mathbb{D}$, the zero product problem is still open. 
In their two papers, Ahern and \v{C}u\v{c}kovi\'{c} showed that $T_fT_g = 0$ 
can only happen in the trivial way provided $f$ and $g$ are bounded harmonic 
functions on $\mathbb{D}$ \cite{AhernCuckovic01} or if $f$ and $g$ are bounded 
radial functions \cite{AhernCuckovic04}. The first result was extended by Choe, Lee, 
Nam and Zheng \cite{ChoeLeeNamZheng07} to the case of the polydisk provided 
the symbols are pluriharmonic. Later Choe and Koo \cite{ChoeKoo06} obtained 
the same conclusion for the unit ball in $\C^n$, where the symbols are bounded 
harmonic functions that have continuous extensions to some open set of the 
boundary.  The second result from \cite{AhernCuckovic04} was extended to 
the unit ball by Dong and Zhou, under the assumption that the symbols are 
(separately)  quasi-homogeneous \cite{DongZhou11}.

So far, all the known results for bounded domains show that the zero product 
problem has a trivial solution. At this point we would like to mention two works 
that study the same problem for unbounded domains.  In an interesting paper by 
\c{C}elik and Zeytuncu \cite{CelikZeytuncu16}, the authors construct an unbounded 
domain in $\C^n$ and a Toeplitz operator that is nilpotent. That is, the zero product 
problem has non-trivial solution.  Similarly, Bauer and Le \cite{BauerLe11} produced 
a curious example of three nonzero Toeplitz operators on the Fock space whose product 
is equal to zero. It is worth noting that in both examples of \cite{CelikZeytuncu16} 
and \cite{BauerLe11}, Toeplitz operators have bounded quasi-homogeneous symbols.

In this paper we study zero products of finitely many Toeplitz operators 
acting on the Bergman space of bounded Reinhardt domains in $\C^n$, with the 
assumption that the symbols are finite sums of quasi-homogeneous functions.  
Our results point again in the direction of the trivial solution. But as opposed 
to the unit ball and the polydisk, working on more general Reinhardt domains 
brings new technical difficulties mainly stemming from the lack of explicit 
formulae for the Bergman kernels on general Reinhardt domains. 
Furthermore, additional difficulties come from 
working with finitely many products of Toeplitz operators and choosing the 
symbols that are finite sums of quasi-homogeneous functions. We hope that our 
techniques will bring more light to the study of algebraic properties of Toeplitz 
operators on more general domains in $\C^n$, beyond the unit ball and polydisk.

Our main result is the following theorem. 

\begin{theorem}\label{ThmMain}
Let $\Omega$ be a bounded Reinhardt domain in $\mathbb C^n$ and  
$\phi_1,\ldots,\phi_m$ be finite sums of bounded quasi-homogeneous functions. 
Assume that $T_{\phi_m}\cdots T_{\phi_1}=0$ on $A^2(\Omega)$. 
Then $\phi_j=0$  for some $j$. 
\end{theorem} 

In particular, our result shows that, the type of examples in  \cite{BauerLe11} and 
\cite{CelikZeytuncu16} cannot happen on bounded domains. In addition, we obtain 
the following corollary.

\begin{corollary}\label{Cor1}
Let $\D$ be a bounded Reinhardt domain in $\C^n$ and $\phi_1,\ldots,\phi_{m-1}$ 
be bounded quasi-homogeneous symbols on $\D$ and let $\phi_m\in L^\infty(\Omega)$. 
Assume that $T_{\phi_m}\cdots T_{\phi_1}=0$ on $A^2(\D)$. Then $\phi_j=0$ 
for some $j$. 
\end{corollary}
In case there are two or more $L^{\infty}$ symbols, the proof of Corollary \ref{Cor1} 
would involve multiple interacting series. Because of this difficulty, the zero product 
problem for Toeplitz operators is still open even on the unit disc for multiple 
$L^{\infty}$ symbols. We believe new ideas are needed in that case.

\section{Preliminaries} 

We start this section by some basic definitions. Let $\D$ be a domain in $\C^n$. 
The space of square integrable holomorphic functions on $\D$, the 
Bergman space on $\D$, is denoted by $A^2(\D)$. Since $A^2(\D)$ is a closed 
subspace of $L^2(\D)$ there exists an orthogonal projection $P:L^2(\D)\to A^2(\D)$, 
called the Bergman projection of $\D$. The Toeplitz operator $T_{\phi}$, with 
symbol $\phi\in L^{\infty}(\D)$, is defined as $T_{\phi}f=P(\phi f)$ for all 
$f\in A^2(\D)$. 

In this paper we are focusing on  bounded Reinhardt domains and products of 
Toeplitz operators with symbols that are finite sums of quasi-homogeneous symbols. 
So we will define these notions next. A domain $\D$ is called Reinhardt if 
$(e^{i\theta_1}z_1,\ldots,  e^{i\theta_n}z_n)\in \D$ whenever $z=(z_1,\ldots, z_n)\in \D$ and 
$\theta_j\in \mathbb{R}$ for all $j$. A function $\phi$ is called (separately) 
quasi-homogeneous if there exists $f:[0,\infty)^n\to \C$, $(k_1,\ldots,k_n)\in \mathbb{Z}^n$ 
such that 
\[\phi(r_1e^{i\theta_1},\ldots,r_ne^{i\theta_n})
=f(r_1, \ldots,r_n)e^{i(k_1\theta_1+\cdots +k_n\theta_n)}.\]  
We note that such functions are called separately quasihomogeneous in \cite{DongZhou11}. 
Next we introduce a condition for a set of multi-indices motivated by \cite{DongZhou11}. 
\begin{definition}\label{de1}
Let $\mathbb{N}_0=\mathbb{N}\cup\{0\}$ and $E$ be a subset of $\mathbb N^n_0$. 
Let $\pi_j$ be the projection of the multi-indices in $\mathbb{N}_0^n$ onto the $j$th coordinate: 
$\pi_j(a_1,\dots,a_n)=a_j$. We  use $E(a_1,\dots,a_j)$ to denote the fiber in $E$ 
with its first $j$ components being $a_1,\dots,a_j$, respectively. That is, 
\[E(a_1,\dots,a_j)=\{\mathbf a\in E:\pi_k(\mathbf a) 
=a_k \text{ for } k=1,2,\dots,j\}.\] 
\begin{itemize}
\item[(i).] We say the fiber $E(a_1,\dots,a_j)$ is \textit{thick} 
if $\pi_{j+1}(E(a_1,\dots,a_j))\cap \mathbb{N}\neq \emptyset$ and 
\[\sum_{k\in \pi_{j+1}(E(a_1,\dots,a_j))\cap \mathbb{N}}\frac{1}{k}=\infty.\] 
We say $E(a_1,\dots,a_j)$ is a \textit{thin} fiber if it is not \textit{thick}.
\item[(ii).] 	We say $E$ satisfies condition (I) if $E$ contains a subset
$\widehat{E}\subset \mathbb{N}^n$ satisfying the following conditions: 
\begin{enumerate} 
\item the sum $\sum_{k\in \pi_1(\widehat{E})}1/k=\infty$, 
\item  for  $1\leq j< n$, and any $j$-tuple $(a_1,\dots,a_j)\in \mathbb N^j$, 
if its corresponding fiber in $\widehat{E}$ 
\[\widehat{E}(a_1,\dots,a_j) 
:=\{\mathbf a\in \widehat{E}:\pi_k(\mathbf a)=a_k \text{ for } k=1,2,\dots,j\}\] 
is nonempty, then $\widehat{E}(a_1,\dots,a_j)$ is thick. 
\end{enumerate}
\end{itemize}
\end{definition}

In case $n=1$, condition (I) above reduces to the well known Blaschke condition 
in the following classical result (see \cite[Page 102]{RemmertBook}).
\begin{theorem}
Let $f$ be a bounded holomorphic function on  right half plane, 
$\mathbb H_+=\{z\in\mathbb C:Re(z)>0\}.$ If the set $E=\{\alpha\in \mathbb{N}:f(\alpha)=0\}$ 
satisfies $\sum_{k\in E\backslash\{0\}}1/k=\infty$, then $f$ vanishes identically on $\mathbb H_+$.
\end{theorem}
Using condition (I) in Definition \ref{de1}, we have the following lemma as 
a higher dimensional analogue of the theorem above. 

\begin{lemma}\label{lem1}
Let $f$ be a bounded holomorphic function on the product of  right half planes 
\[\mathbb H_+^n=\{z\in \mathbb C^n:\re (z_j)>0, j=1,\dots,n\}.\]
If the set $E=\{\alpha\in \mathbb{N}^n:f(\alpha)=0\}$ satisfies 
condition (I), then $f$ vanishes identically on $\mathbb H_+^n$.
\end{lemma}

\begin{proof} 
Since $E$ satisfies condition (I) there exists a subset $\widehat{E}$ of 
$E\cap \mathbb{N}^n$ such that for any fixed multi-index $(a_1,\dots,a_{n-1})$ 
with $\widehat{E}(a_1,\dots,a_{n-1})\neq \emptyset$, there exists a sequence 
$\{a^{(l)}_n\}$ such that $(a_1,\dots,a_{n-1},a^{(l)}_n)\in \widehat{E}$  for each $l$ 
and $\sum_{l=1}^{\infty}1/a^{(l)}_n=\infty$. Then $f(a_1,\dots,a_{n-1},z_n)\equiv0$ 
on $\mathbb H_+^1$ for any fixed multi-index $(a_1,\dots,a_{n-1})$ with 
$\widehat{E}(a_1,\dots,a_{n-1})\neq \emptyset$ (see \cite[Page 102]{RemmertBook} 
and  \cite[Theorem 2.3]{DongZhou11}). Since  
\[\widehat{E}(a_1,\dots,a_{n-2})\supset 
 \widehat{E}(a_1,\dots,a_{n-2},a_{n-1}) \neq \emptyset,\] 
there also exists a sequence $\{a^{(l)}_{n-1}\}$ such that 
$\widehat{E}(a_1,\dots,a_{n-2},a^{(l)}_{n-1})\neq \emptyset$ 
and $\sum_{l=1}^{\infty}1/a^{(l)}_{n-1}=\infty$. Thus for every fixed $z_n\in \mathbb H_+^1$
 and every $(a_1,\dots,a_{n-2})$ with $\widehat{E}(a_1,\dots,a_{n-2})\neq \emptyset$ we have  
$f(a_1,\dots,a_{n-2},z_{n-1},z_n)\equiv0$ on $\mathbb H_+^1$. Repeating this process yields 
$f(z)\equiv0$ on $\mathbb H^n_+$.
\end{proof}
\begin{lemma}\label{lem2}
Let $Z_1$ and $Z_2$ be subsets of $\mathbb N^n_0$ such that $Z_1\cup Z_2$ 
satisfies condition (I). Then $Z_1$ or $Z_2$ satisfy condition (I).
\end{lemma}

\begin{proof}
Let  $M=Z_1\cup Z_2$. Since $M$ satisfies condition (I), there exists 
$\widehat{M}\subset (Z_1\cup Z_2)\cap \mathbb{N}^n$  satisfying 
(ii) in Definition \ref{de1}. Assume that $Z_1$ does not satisfy 
condition (I). Then we apply the following process to  $Z_1$:
\begin{itemize}
\item[i.] Set $E_{n}=Z_1$.
\item[ii.] Define $E_{n-1}$ to be the set obtained by deleting all 
thin fibers of the form $E_n(a_1,\dots,a_{n-1})$ from $E_n$, i.e.
\[E_{n-1}=\{(a_1,\dots,a_n)\in E_n:E_n(a_1,\dots,a_{n-1}) \text{ is thick}\}.\] 
\item[iii.] For $1\leq j\leq n-2$, define $E_j$ to be the set obtained by deleting 
all thin fibers of the form $E_{j+1}(a_1,\dots,a_j)$ from $E_{j+1}$, i.e.  
\[E_j=\{(a_1,\dots,a_n)\in E_{j+1}:E_{j+1}(a_1,\dots,a_j) \text{ is thick}\}.\]
\item[iv.] If $E_1\neq \emptyset$, then we set 
$E_0=\begin{cases}
	\emptyset  & \text{ if } \sum_{j\in \pi_1(E_1)}{1}/{j}<\infty,\\ 
E_1 & \text{ otherwise.} 
	\end{cases}$ 
\end{itemize}

By running the above process, all the thin fibers  will be deleted from $Z_1$. 
That is, in step ii. all the remaining non-empty fibers $E_{n-1}(a_1,\ldots,a_{n-1})$ 
will be thick. In step iii. we remove (two dimensional) thin (in the $(n-1)$-st component) 
fibers of the form $E_{n-1}(a_1,\ldots,a_{n-2})$ to get the set $E_{n-2}$. Therefore, if 
$E_{n-2}(a_1,\ldots,a_{n-2})$ is non-empty then it is thick and 
\[E_{n-2}(a_1,\ldots,a_{n-2})=E_{n-1}(a_1,\ldots,a_{n-2}).\] 
When deleting  the two dimensional  fibers from $E_{n-1}$ to obtain $E_{n-2}$, 
a one dimensional fiber $E_{n-1}(a_1,\ldots,a_{n-1})$ in $E_{n-1}$  would either 
be untouched or entirely removed. In other words, if the (one dimensional) 
fiber $E_{n-2}(a_1,\ldots,a_{n-1})$ is non-empty then it is thick and 
\[E_{n-2}(a_1,\ldots,a_{n-1})=E_{n-1}(a_1,\ldots,a_{n-1}).\] 
Therefore, the non-empty fibers $E_{n-2}(a_1,\ldots,a_k)$ are thick for $n-2\leq k\leq n-1$. 
Arguing inductively in step iii., we conclude that the non-empty 
fibers $E_j(a_1,\ldots,a_{k})$ are thick for $1\leq j\leq k\leq n-1$.
It is possible that $E_j$ becomes the empty set for some $j\geq 1$. 
Nevertheless, $E_0$ is either an empty set, or a set with all 
fibers in it being thick  and satisfy $\sum_{j\in \pi_1(E_0)}\frac{1}{j}=\infty.$ 
Since $Z_1$ does not satisfy condition (I), $E_0$ has to be an empty set. 
Thus $Z_1=\cup_{j=1}^n(E_{j}\backslash E_{j-1})$. Set $F_j=E_j\backslash E_{j-1}$ 
for $j=1,\ldots,n$. Then we have the following properties for $F_j$: 
\begin{itemize}
\item [i.] $Z_1=\cup_{j=1}^nF_j$.
\item [ii.] $F_1=E_1$ where $E_1$ either satisfies $\sum_{j\in \pi_1(E_1)}1/j<\infty$ or it is empty.
\item [iii.] For $2\leq j\leq n$, $F_j$ consists of all thin 
fibers in $E_j$ of the form $E_j(a_1,\dots,a_{j-1})$, i.e.
\[F_j=E_j\backslash E_{j-1} 
=\{(a_1,\dots,a_n)\in E_j:E_j(a_1,\dots,a_{j-1}) \text{ is thin}\} \;\;\text{ for }\;\; 2\leq j\leq n.\]
\end{itemize}

Note that if  a fiber $\widehat{M}(a_1,\dots,a_{n-1})$ of $\widehat{M}$ is thick,  
then $\widehat{M}(a_1,\dots,a_{n-1})\backslash F_n(a_1,\dots,a_{n-1})$ is still thick 
since  $F_n(a_1,\dots,a_{n-1})$  is thin. Hence deleting $F_n$ from $\widehat{M}$ 
does not affect the ``thickness" of 
$\widehat{M}(a_1,\dots,a_{n-1})\backslash F_n(a_1,\dots,a_{n-1})$, fibers with their 
first $(n-1)$ components fixed. Moreover, all non-empty fibers of the form 
$\widehat{M}(a_1,\dots,a_{n-2})\backslash F_n(a_1,\dots,a_{n-2})$ are still thick,  
because a thick portion of the fiber survives when we remove a thin fiber. 
To be more precise, if $\widehat{M}(a_1,\dots,a_{n-1})$ is non-empty then it is thick. 
Furthermore, since $F_n(a_1,\dots,a_{n-1})$ is thin, the set 
$\widehat{M}(a_1,\dots,a_{n-1})\backslash F_n(a_1,\dots,a_{n-1})$ 
is non-empty and thick. That is, 
\[\pi_{n-1}(\widehat {M}(a_1,\dots,a_{n-2}))
=\pi_{n-1}(\widehat {M}(a_1,\dots,a_{n-2})\backslash F_n).\] 
Thus the thickness of the fiber corresponding to the $(n-2)$-tuple 
$(a_1,\dots,a_{n-2})$ stays the same. Similarly the thickness of fibers 
with their first $j$ components fixed are not affected for $j\leq n-3$. 
Thus $\widehat M\backslash F_n$ is still a set satisfying (ii) in Definition \ref{de1}. 

Now we turn to show that the set  $(\widehat M\backslash F_n)\backslash F_{n-1}$ 
contains a subset satisfying (ii) of Definition \ref{de1}. The same argument as in 
the previous paragraph yields that, for $j\leq n-2$,  
none of the fibers  with their first $j$ components fixed are entirely deleted from 
$\widehat M\backslash F_n$ and hence all fibers in 
$(\widehat M\backslash F_n)\backslash F_{n-1}$ with their first $j$ components fixed are thick. 
For each $(n-2)$-tuple $\mathbf a^\prime=(a_1,\dots, a_{n-2})$ such that 
$(\widehat M\backslash F_n)(\mathbf a^\prime)\backslash F_{n-1}(\mathbf a^\prime)$ is 
nonempty, we have $(\widehat M\backslash F_n)(\mathbf a^\prime)$  is thick and 
$F_{n-1}(\mathbf a^\prime)$ is thin. Thus 
\[\sum_{j\in \pi_{n-1} ((\widehat M\backslash F_n)(\mathbf a^\prime)
\setminus F_{n-1}(\mathbf a^\prime))}\frac{1}{j}=\infty.\] 
Moreover, for each 
$j\in \pi_{n-1} ((\widehat M\backslash F_n)(\mathbf a^\prime)
\setminus F_{n-1}(\mathbf a^\prime)) $, 
the fiber $((\widehat M\backslash F_n)\backslash F_{n-1})(\mathbf a^\prime,j)$ is thick 
since $F_{n-1}(\mathbf a^\prime,j)$ is empty. Therefore the set 
\[\widehat M_{n-1}:=\{(a_1,\dots,a_n)\in(\widehat M\backslash F_n)
\backslash F_{n-1}: F_{n-1}(a_1,\dots,a_{n-1})=\emptyset\}\]
is a subset of  $(\widehat M\backslash F_n)\backslash F_{n-1}$ satisfying (ii) of  
Definition \ref{de1} which implies that $(\widehat M\backslash F_n)\backslash F_{n-1}$  
satisfies condition (I). 

Similarly, we can show that
$((\widehat{M}\backslash F_n)\backslash F_{n-1})\backslash F_{n-2}$, $\dots$, up to the 
set $\widehat{M}\backslash (\cup_{l=1}^nF_l)$ all satisfy condition (I). Hence 
$\widehat{M}\backslash Z_1$ satisfies condition (I). It follows from the containment  
$Z_2\supseteq \widehat{M}\backslash Z_1$ that condition (I) holds true for the set $Z_2$. 
\end{proof}
As a direct consequence of Lemma \ref{lem1} we have the following corollary. 
\begin{corollary}\label{Cor2}
Let $Z_1, \ldots,Z_m$ be subsets of $\mathbb N^n_0$ such that $\cup_{j=1}^mZ_j$  
satisfies condition (I). Then $Z_j$ satisfies condition (I) for some $1\leq j\leq m$. 
\end{corollary}

\section{Proof of Corollary \ref{Cor1} and Theorem \ref{ThmMain}}
We start this section with a simpler version of Theorem \ref{ThmMain}. 
\begin{proposition}\label{Prop1}
Let $\D$ be a bounded Reinhardt domain in $\C^n$ and $\phi_1,\ldots,\phi_m$ 
be bounded quasi-homogeneous symbols on $\D$. Assume that 
$T_{\phi_m}\cdots T_{\phi_1}=0$ on $A^2(\D)$. Then $\phi_j=0$ for some $j$.
\end{proposition}

\begin{proof}
Let $\Omega^+$ and $\widetilde{\Omega}^+$ denote the subsets in 
$\mathbb R^n$ defined by 
\begin{align*}
\Omega^+ =&\{(|z_1|,\dots,|z_n|)\in \mathbb R^n:z=(z_1,\dots,z_n)\in \Omega\};\\
\widetilde{\Omega}^+ 
=&\{(x^2_1,\dots,x^2_n)\in \mathbb R^n:(x_1,\dots,x_n)\in \Omega^+\}.
\end{align*}	
Since $\phi_1,\ldots,\phi_m$  are quasi-homogeneous, for  
$z=(r_1e^{i\theta_1},\dots,r_ne^{i\theta_n})\in \D$ and each $j$ we set 
$\phi_j(z)=f_j(r)e^{i\mathbf k_j\cdot \theta}$ where $r=(r_1,\dots,r_n)$, 
$\theta=(\theta_1,\dots,\theta_n)$, and $\mathbf k_j\in \mathbb Z^n$. 
The Bergman kernel function $K_\Omega$ has the expression
\[K_\Omega(z,w)=\sum_{\alpha\in \mathbb Z^n}c_\alpha z^\alpha\bar w^\alpha\]
where $c_\alpha=
\begin{cases}
0& \|z^\alpha\|=\infty\\\|z^{\alpha}\|^{-2}& \text{ otherwise.}
	\end{cases}$ 
A polar coordinates computation yields for $\mathbf k\in \mathbb N^n$ that
\begin{align}\label{1}
T_{\phi_1}z^{\mathbf k} 
&=\int_\Omega \sum_{\alpha\in \mathbb Z^n}
c_\alpha z^\alpha\bar w^\alpha \phi_1(w)w^{\mathbf k}dV(w)
\nonumber\\&= (2\pi)^nc_{\mathbf  k+\mathbf  k_1}
z^{\mathbf  k+\mathbf  k_1}\int_{\Omega^+} f_1(r)
r^{2\mathbf k+\mathbf k_1+\mathbf 1}dV(r)
\nonumber\\&=\pi^n c_{\mathbf  k+\mathbf  k_1}
z^{\mathbf  k+\mathbf  k_1}\int_{\widetilde\Omega^+} 
f_1(\sqrt{t})t^{\mathbf k+(\mathbf k_1/2)}dV(t),
\end{align}
where $\sqrt{t}=(\sqrt{t_1},\dots,\sqrt{t_n})$. Set 
$g_1(t)=f_1(\sqrt{t})t^{\mathbf k_1/2}$ and 
$d_1(\mathbf k)=\pi^n c_{\mathbf  k+\mathbf  k_1}$. Then \eqref{1}
 becomes 
\[T_{\phi_1}z^{\mathbf k} =d_1(\mathbf k)
z^{\mathbf k+\mathbf k_1}\int_{\widetilde \Omega^+}g_1(t)t^{\mathbf k}dV(t).\] 
Similarly,
\begin{align}\label{2}T_{\phi_2}z^{\mathbf k+\mathbf k_1} 
&=\int_\Omega \sum_{\alpha\in \mathbb Z^n}
c_\alpha z^\alpha\bar w^\alpha \phi_2(w)w^{\mathbf k+\mathbf k_1}dV(w)
\nonumber\\&=\pi^n c_{\mathbf  k+\mathbf  k_1+\mathbf k_2}
z^{\mathbf  k+\mathbf  k_1+\mathbf k_2}\int_{\widetilde\Omega^+} 
f_2(\sqrt{t})t^{\mathbf k+\mathbf k_1+(\mathbf k_2/2)}dV(t).
\end{align}
 We also set $g_2(t)= f_2(\sqrt{t})t^{\mathbf k_1+(\mathbf k_2/2)}$ and 
$d_2(\mathbf k)=\pi^nc_{\mathbf  k+\mathbf  k_1+\mathbf k_2}$. Then 
\[T_{\phi_2}T_{\phi_1}z^{\mathbf k} 
=d_1(\mathbf k)d_2(\mathbf k)z^{\mathbf k+\mathbf k_1+\mathbf k_2}
\int_{\widetilde \Omega^+}g_1(t)t^{\mathbf k}dV(t)
\int_{\widetilde \Omega^+}g_2(t)t^{\mathbf k}dV(t).\]
 Repeating this process up to $T_{\phi_m}$, we obtain 
 \[T_{\phi_m}\cdots T_{\phi_1}z^{\mathbf k} 
=z^{\mathbf k}\prod_{j=1}^{m}z^{\mathbf k_j}d_j(\mathbf k)
\int_{\widetilde \Omega^+}g_j(t)t^{\mathbf k}dV(t),\]
 where $d_j(\mathbf k)=\pi^nc_{\mathbf k+\sum_{l=1}^{j}\mathbf k_l}$ 
and $g_j(t)=f_j(\sqrt{t})t^{\mathbf k_1+\cdots+\mathbf k_{j-1}+(\mathbf k_j/2)}$. 
The assumption  $$T_{\phi_m}\cdots T_{\phi_1}=0$$ on $A^2(\Omega)$ then 
implies the equation 
\begin{equation}\label{3}
z^{\mathbf k}\prod_{j=1}^{m}z^{\mathbf k_j}d_j(\mathbf k)
\int_{\widetilde \Omega^+}g_j(t)t^{\mathbf k}dV(t)=0,
\end{equation} 
for all $z^{\mathbf k}\in A^2(\Omega)$. 

Since $\Omega$ is bounded, we have 
$A^2(\Omega)\supseteq\{z^{\mathbf k}\}_{\mathbf  k\in \mathbb N^n}$. 
Moreover, there exists a multi-index $\mathbf k_0\in \mathbb N^n$ 
such that for any $\mathbf k\geq \mathbf k_0$ and any $j$ the constant 
$d_j(\mathbf k)>0$. Thus from \eqref{3} we have 
 \begin{equation}\label{4}
\prod_{j=1}^{m}\int_{\widetilde \Omega^+}g_j(t)t^{\mathbf k}dV(t)=0 
\end{equation}
for any $\mathbf k\geq \mathbf k_0$. Equivalently, 
 \begin{equation}\label{5}
\prod_{j=1}^{m}\int_{\widetilde \Omega^+}(g_j(t)t^{\mathbf k_0+\mathbf 1})
t^{\mathbf k}dV(t)=0
\end{equation}
for any $\mathbf k\in \mathbb N^n$.  Set 
\[Z_j=\left\{\mathbf k\in \mathbb N^n: 
\int_{\widetilde \Omega^+}(g_j(t)t^{\mathbf k_0})t^{\mathbf k}dV(t)=0\right\}.\] 
Then \eqref{5} implies that 
\[\cup_{j=1}^m Z_j=\mathbb{N}^n.\]
Clearly $\mathbb{N}^n$ satisfies condition (I). Then by Corollary \ref{Cor2}, $Z_{j_0}$ 
satisfies condition (I) for some $j_0$. After a change of variables $t=cx$ for sufficiently 
small $c>0$, we may further assume that $\widetilde{\Omega}^+\subseteq [0,1)^n$. 
Then the function 
\[h_j(z)=\int_{\widetilde \Omega^+}(g_j(t)t^{\mathbf k_0})t^{z}dV(t)\] 
is bounded and holomorphic on $\mathbb H^n_+$. Then Lemma  \ref{lem1} implies that 
$h_{j_0}\equiv 0$. Thus $Z_{j_0}=\mathbb N^n$ and, Stone-Weierstrass Theorem 
implies that  $g_{j_0}(t)t^{\mathbf k_0}\equiv0$. Therefore $f_{j_0}\equiv 0$ 
which proves the proposition. 
\end{proof}
As a consequence of  Proposition \ref{Prop1}, we also have the following 
 proof of Corollary \ref{Cor1}.

\begin{proof}[Proof of Corollary \ref{Cor1}] 
Note that the bounded function $\phi_m$  has an $L^2$ expansion
\[\phi_m(r_1e^{i\theta_1},\ldots,r_ne^{i\theta_n}) 
=\sum_{\mathbf p\in \mathbb Z^n}f_{m,\mathbf p}(r_1,\ldots,r_n)e^{i\mathbf p\cdot\theta}\] 
where 
\[f_{m,\mathbf p}(r_1,\dots,r_n)
=\frac{1}{(2\pi)^n}\int_0^{2\pi}\cdots\int_0^{2\pi}
\phi_m(r_1e^{i\theta_1},\ldots,r_ne^{i\theta_n}) 
e^{-i\mathbf p\cdot \theta}d\theta_1\cdots d\theta_n\]
are also bounded (see, for example, \cite[Lemma 2.2]{Le10}  or 
\cite[Lemma 4.1]{DongZhou11}). From the assumption, the symbol functions 
$\phi_1,\dots, \phi_{m-1}$ are bounded and quasi-homogeneous. Thus we 
can set $\phi_j(z)=f_j(r)e^{i\mathbf{k}_j\cdot\theta}$ for  $1\leq j\leq m-1$ 
where $r=(r_1,\dots,r_n)$, 	$\theta_j=(\theta_1,\dots,\theta_n)$, and 
$\mathbf k_j\in \mathbb Z^n$. A similar computation as in the 
proof of Proposition \ref{Prop1} then shows that for any $z^\mathbf k$ 
with $\mathbf k\in \mathbb N^n$, 
\begin{align}\label{EqnCor}
T_{\phi_m}\cdots T_{\phi_1}z^{\mathbf k} 
= &\sum_{\mathbf p\in \mathbb Z^n} z^{\mathbf k+\mathbf p}
d_{m,\mathbf p}(\mathbf k)\int_{\widetilde \Omega^+}g_{m,\mathbf p}(t)
t^{\mathbf k}dV(t)\prod_{j=1}^{m-1}z^{\mathbf k_j}d_j(\mathbf k)
\int_{\widetilde \Omega^+}g_j(t)t^{\mathbf k}dV(t)\\ 
\nonumber =&0,
\end{align}
where 
\begin{align*}
d_{m,\mathbf p}(\mathbf k)=&\pi^nc_{\mathbf k+\mathbf p
+\sum_{l=1}^{m-1}\mathbf k_l},\\
d_j(\mathbf k)=&\pi^nc_{\mathbf k+\sum_{l=1}^{j}\mathbf k_l},\\
g_j(t)=&f_j(\sqrt{t})t^{\mathbf k_1+\cdots+\mathbf k_{j-1}+(\mathbf k_j/2)},\\
g_{m,\mathbf p}(t)=&f_{m,\mathbf p}(\sqrt{t})t^{\mathbf k_1
+\cdots+\mathbf k_{m-1}+(\mathbf p/2)}.
\end{align*}
The first equality in \eqref{EqnCor} holds by the $L^2$ convergence of the series  
$\sum_{\mathbf p\in \mathbb Z^n}f_{m,\mathbf p}(r)e^{i\mathbf p\cdot\theta}$ 
and the boundedness of $z^\mathbf k$. This together with the fact 
that $\mathbf{k}_1,\ldots,\mathbf{k}_{m-1}$ are fixed imply that for any 
$\mathbf k\in \mathbb N^n$ and $\mathbf p\in \mathbb Z^n$ we have 
\[z^{\mathbf k+\mathbf p}d_{m,\mathbf p}(\mathbf k)
\int_{\widetilde \Omega^+}g_{m,\mathbf p}(t)t^{\mathbf k}dV(t)
\prod_{j=1}^{m-1}z^{\mathbf k_j}d_j(\mathbf k)
\int_{\widetilde \Omega^+}g_j(t)t^{\mathbf k}dV(t)=0.\]
Then the argument following \eqref{3} in the proof of Proposition \ref{Prop1} 
yields that either $\phi_j=0$ for some $j\in {1,2,\dots m-1}$ or $f_{m,\mathbf p}=0$. 
If $\phi_j=0$ for some $j\in {1,2,\dots m-1}$, then we are done. Otherwise, 
$f_{m,\mathbf p}=0$ for all $\mathbf p\in \mathbb Z^n$ which yields that $\phi_m=0$.
\end{proof}
Now we turn our attention to proving Theorem \ref{ThmMain}. 
We introduce the following notation which will be used in the proof. 
For multi-indices $\mathbf{a}=(a_1,\ldots,a_n)$ and 
$\mathbf{b}=(b_1,\ldots,b_n)$ we say $\mathbf{a}\leq \mathbf{b}$ if 
$a_j\leq b_j$ and  $\mathbf{a}<  \mathbf{b}$ if $a_j< b_j$ for all $j$. 
When $\mathbf{a}< \mathbf{b}$ the \textit{box}  
\[R=\{\mathbf k\in \mathbb Z^n: a_j\leq k_j<b_j \text{ for }j=1,2,\dots,n\}\]
is said to have dimension $\mathbf{d}=(d_1,\ldots, d_n)$ where $d_j=b_j-a_j$  
for $j=1,2,\dots,n$. That is, $d_j$ is the size of $R$ in the $j$th component. 

\begin{proof}[{Proof of Theorem \ref{ThmMain}}] 
Since $\phi_j$s are finite sums of bounded quasi-homogeneous functions, 
we may assume that for $j=1,\ldots, m$
\begin{align*}
\phi_j(z)&=\sum_{\mathbf k\in R_j}\phi_{j,\mathbf k}(z).
\end{align*}
where $\phi_{j,\mathbf k}(z)=f_{j,\mathbf k}(r)e^{i\mathbf k\cdot \theta}$
are bounded 
quasi-homogeneous functions and $R_j$ is a box in $\mathbb Z^n$ of dimensions 
$\mathbf d(j)=(d_1(j),\dots,d_n(j))$. We will prove the theorem using 
an induction on the dimensions of $R_1,\ldots, R_m$.

Let $\mathbf 1_l$ denote the multi-index in $\mathbb Z^n$ whose $l$th entry 
is 1 and every other entry is 0. When all of $R_j$s are of dimension at 
most $(1,\dots,1)$, the symbols $\phi_j$  are bounded quasi-homogeneous. 
Then the statement of Theorem  \ref{ThmMain} follows by Proposition \ref{Prop1}. 

Next, suppose the statement is proved when all of $R_j$s have dimensions less than 
or equal to $\mathbf d(j)\geq (1,\ldots,1)$. We claim the same result holds true 
when  for some $j_0$, $R_{j_0}$ has dimensions $\mathbf d^\prime(j_0)$ where
$\mathbf d^\prime (j_0)=\mathbf d(j_0)+\mathbf 1_l$ while the 
dimensions of all  other $R_j$s are still  less than or equal to 
$\mathbf d(j)\geq (1,\ldots,1)$.

Without loss of generality, we may assume that   $l=1$ as the other 
cases can be proven by the exact same argument. Since 
$\phi_j(z)=\sum_{\mathbf k\in R_j}\phi_{j,\mathbf k}(z)$, we have 
\begin{align*}
T_{\phi_m} \cdots T_{\phi_1}=\sum_{\mathbf k_1\in R_1}\cdots\sum_{\mathbf k_m\in R_m}
T_{\phi_{m,\mathbf k_m}}\cdots T_{\phi_{1,\mathbf k_1}}.
\end{align*}
Using the same notation as in the proof of Proposition \ref{Prop1}, it follows that for 
$z^{\mathbf k}\in A^2(\Omega)$,
\begin{align}
0=&\sum_{\mathbf k_1\in R_1}\cdots\sum_{\mathbf k_m\in R_m}
T_{\phi_{m,\mathbf k_m}}\cdots T_{\phi_{1,\mathbf k_1}}z^{\mathbf k}\nonumber\\ 
\label{6}=&\sum_{(\mathbf  k_1,\dots,\mathbf k_m)\in R_1\times\cdots\times R_m} 
z^{\mathbf k}\prod_{j=1}^{m}z^{\mathbf k_j}d_j(\mathbf k)
\int_{\widetilde \Omega^+}g_{j,\mathbf k_j}(t)t^{\mathbf k}dV(t).
\end{align}
Suppose $R_s=\{\alpha\in \mathbb Z^n: a_{sj}\leq\alpha_j< b_{sj}+1 \text{ for } 1\leq j\leq n\}$ 
for $s=1\ldots,m$. Then the dimensions of $R_s$s imply that  
$b_{j_01}-a_{j_01}=d^\prime_1(1)-1=d_1(j_0)$ and $b_{s1}-a_{s1}=d_1(s)-1$ for $s\neq j_0$.
Thus the monomial $z^{\mathbf k}\prod_{j=1}^{m}z^{\mathbf k_j}$ contains the  
factor $z_1^{k_1}\prod_{j=1}^{m}z_1^{b_{j1}}$ only if the first entries of $\mathbf k_j$ 
are $b_{j1}$ respectively (because $\alpha_1\leq b_{j1}$ when $\alpha\in R_j$). We set 
\begin{align*}
R_j^1&=\{\alpha\in R_j: \alpha_{1}=b_{j1}\},
\end{align*} 
and define  
\[\phi^1_j(z)=\sum_{\mathbf k\in R^1_j}\phi_{j,\mathbf k}(z).\]
Then \eqref{6} implies that 
\begin{align}\label{Eqn6-1}
\sum_{(\mathbf k_1,\ldots,\mathbf k_m)\in R^1_1\times \cdots\times R^1_m} 
z^{\mathbf k}\prod_{j=1}^{m}z^{\mathbf k_j}d_j(\mathbf k)
\int_{\widetilde \Omega^+}g_{j,\mathbf k_j}(t)t^{\mathbf k}dV(t)=0,
\end{align}
or equivalently $T_{\phi^1_m}\cdots T_{\phi^1_1}=0$ on $A^2(\Omega)$, since 
the terms  in \eqref{Eqn6-1} have the same factor 
$z_1^{k_1}\prod_{j=1}^{m}z_1^{b_{j1}}$, the largest 
power of  $z_1$ in \eqref{6}. Because $R^1_j$ has dimensions less than 
or equal to $\mathbf d(j)$, the induction hypothesis  implies that 
$\phi^1_j=0$ for some $j$. 

If $\phi^1_{j_0}=0$, then 
\[\phi_{j_0}(z)=\sum_{\mathbf p\in R_{j_0}\backslash R^1_{j_0}}\phi_{j_0,\mathbf p}(z).\] 
Note that  $R_{j_0}\backslash R^1_{j_0}$ is a box of dimensions  equal to 
$\mathbf d(j_0)$ and hence the induction hypothesis then yields that $\phi_j=0$ 
for some $j$. On the other hand, if $\phi^1_{j_1}=0$ for some 
$j_1\neq j_0$, then the dimensions of $R_{j_1}$ is reduced and 
\[\phi_{j_1}(z)=\sum_{\mathbf p\in R_{j_1}\backslash R^1_{j_1}}\phi_{j_1,\mathbf p}(z)\] 
where 
\[R_{j_1}\backslash R_{j_1}^1=\{\alpha\in R_{j_1}: a_{j_11}\leq \alpha_{1}< b_{j_11} 
\}.\]
We set 
$R^2_{j_1}=\{\alpha\in R_{j_1}:  \alpha_{1}= b_{j_11}-1 \}$ and define  
\[\phi^2_{j_1}(z) =\sum_{\mathbf k\in R^2_{j_1}}\phi_{j_1,\mathbf k}(z).\] 

Now we consider the equation 
$T_{\phi^1_m}\cdots T_{\phi^2_{j_1}}\cdots T_{\phi^1_1}=0$ on $A^2(\Omega)$.
By the same argument, we have either $\phi^2_{j_1}=0$ or $\phi^1_j=0$ for some 
$j\neq j_1$. Again, if $\phi^1_{j_0}=0$, then it follows, from the induction 
hypothesis applied to $T_{\phi_m} \cdots T_{\phi_1}$,  
that $\phi_j=0$ for some $j$. If $\phi^1_j=0$ for some $j\neq j_0, j_1$ 
then $R_j$ can be reduced to $R_j\backslash R^1_j$. If  $\phi^2_{j_1}=0$, 
then the dimension of $R_{j_1}$ will be further reduced. 

Repeating this process, we would either have $\phi^1_{j_0}=0$ at some step 
which completes the proof, or the process continues. At every 
step, one of $R_j$s will have its dimensions decreased by $1$ along certain 
coordinate direction. Since each $R_j$ has finite dimensions, some $R_j$ 
will be reduced to an empty set  after finitely many steps. This would imply 
that the corresponding symbol $\phi_j=0$ which also completes the proof. 
\end{proof}


\end{document}